\documentclass{amsart}
\usepackage{amsmath,amsthm}
\usepackage{amsfonts,amssymb}
\usepackage{enumerate}
\usepackage{graphicx}

\def\PP{{\mathcal P}}

\def\wt{\widetilde}

\newtheorem{lemma}{Lemma}[section]
\newtheorem{proposition}[lemma]{Proposition}

\newtheorem{theorem}[lemma]{Theorem}
\newtheorem{remark}[lemma]{Remark}

\makeatletter
\@addtoreset{equation}{section}
\makeatother
\def\qed{\hfill $\Box$\smallskip}

\newcommand {\ds} {\displaystyle}
\newcommand {\sgn} {{\rm sign}\,}
\begin{document}

\title[Discrete weighted Markov--Bernstein inequality]
{A discrete weighted Markov--Bernstein inequality for polynomials
and sequences}


\author{Dimitar K. Dimitrov}
\address{Departamento de Matem\'atica Aplicada, IBILCE,
Universidade Estadual Paulista, 15054-000
S\~{a}o Jos\'{e} do Rio Preto, SP, Brazil }
\email{d\underline{\space}k\underline{\space}dimitrov@yahoo.com}

\author{Geno P. Nikolov}
\address{Faculty of Mathematics and Informatics, Sofia University
"St. Kliment Ohridski", 5 James Bourchier Blvd.,
1164 Sofia, Bulgaria}
\email{geno@fmi.uni-sofia.bg}

\thanks{The first author is supported by the Brazilian foundations
CNPq under Grant 306136/2017--1 and FAPESP under Grants
2016/09906--0 and 2016/10357-1. The second author is supported in
part by the Bulgarian National Research Fund through Contract DN
02/14 and by Sofia University Research Fund through Contract
3067/2020.}

\begin{abstract}
For parameters $\,c\in(0,1)\,$ and $\,\beta>0$,  let $\,\ell_{2}(c,\beta)\,$
be the Hilbert space of real functions defined on $\,\mathbb{N}\,$
(i.e., real sequences),  for which
$$
\| f \|_{c,\beta}^2 := \sum_{k=0}^{\infty}\frac{(\beta)_k}{k!}\,c^k\,[f(k)]^2<\infty\,.
$$
We study the best (i.e., the smallest possible) constant $\,\gamma_n(c,\beta)\,$
in the discrete Markov-Bernstein inequality
$$
\|\Delta P\|_{c,\beta}\leq \gamma_n(c,\beta)\,\|P\|_{c,\beta}\,,\quad P\in\PP_n\,,
$$
where $\,\PP_n\,$ is  the set of real algebraic polynomials of
degree at most $\,n\,$ and $\,\Delta f(x):=f(x+1)-f(x)\,$.

We prove that
\begin{itemize}
\item[(i)]
$\ds \gamma_n(c,1)\leq 1+\frac{1}{\sqrt{c}}\,$ for every $\,n\in \mathbb{N}\,$ and
$\ds \lim_{n\to\infty}\gamma_n(c,1)= 1+\frac{1}{\sqrt{c}}\,$;
\item[(ii)]
For every fixed $\,c\in (0,1)\,$, $\,\gamma_n(c,\beta)\,$ is a monotonically
decreasing function of $\,\beta\,$ in $\,(0,\infty)\,$;
\item[(iii)]
For every fixed  $\,c\in (0,1)\,$ and $\,\beta>0\,$, the best Markov-Bernstein
constants $\,\gamma_n(c,\beta)\,$ are bounded uniformly with respect to $\,n$.
\end{itemize}
A similar Markov-Bernstein unequality is proved for sequences in
$\,\ell_{2}(c,\beta)\,$. We also establish a relation between the
best Markov-Bernstein constants $\,\gamma_n(c,\beta)\,$ and the
smallest eigenvalues of certain explicitly given Jacobi matrices.
\end{abstract}

\maketitle

\section{Introduction and statement of the results}

Throughout this paper $\PP_n$ and $\PP_n^{\mathbb{C}}$ stands for the set of
real and complex algebraic polynomials of degree not exceeding $n$ and $\PP$
for all real polynomials. The inequalities
of the form
\begin{equation}\label{e1.1}
\Vert p' \Vert\leq c_{n} \Vert p\Vert,\quad p\in\PP_n\quad \mathrm{or}\quad
 p\in\PP_n^{\mathbb{C}},
\end{equation}
which hold for various norms are called Markov--Bernstein--type inequalities.
Andrey Markov \cite{AM1889} settled the classical case of real polynomials and
the uniform norm in $[-1,1]$. Precisely, he
showed that in this case the Chebyshev polynomial of the first kind,
$T_n(x)=\cos n\arccos x$, $x\in [-1,1]$, is the only (up to a
constant factor) extremal polynomial and the best, that is, the
smallest possible constant $c_{n}$ is equal to $T_n^\prime(1)=n^2$.
Later E. Hille, G. Szeg\H{o} and J. D. Tamarkin \cite{HST} proved inequality  (\ref{e1.1})
for the norm in $L^p[-1,1]$,  $1\leq p <\infty$. The inequality
\begin{equation}\label{BI}
\Vert p' \Vert_{L^p(\partial \mathbb{D})}\leq n \,\Vert p\Vert_{L^p(\partial \mathbb{D})},\quad p\in\PP_n^{\mathbb{C}},
\end{equation}
where, for $0\leq p \leq\infty$,
$$
\Vert f \Vert_{L^p(\partial \mathbb{D})} = \left( \int_{-\pi}^\pi | f(e^{i\theta}) |^p d \theta \right)^{1/p},
$$
holds for every $p>0$. It is usually called the Bernstein inequality though the first proof for the case
$p=\infty$ is due to M. Riesz \cite{MR} (see \cite{Nev2014}). It was then established for $1\leq p < \infty$ and
V. Arestov \cite{Ar} settled the case $p\in (0,1)$. Similar inequalities hold for entire functions. Indeed,
if $f$ is an entire function of exponential type $\sigma$, such that $f \in L^p(\mathbb{R})$, then
\begin{equation}
\label{BEF}
\Vert f' \Vert_{L^p(\mathbb{R})}\leq \sigma \,\Vert f\Vert_{L^p(\mathbb{R})}.
\end{equation}
We refer to \cite[Theorem 11.3.3]{Boas} and \cite{RS1990} for the cases $1\leq p \leq \infty$ and
$p\in (0,1)$, respectively.

Various weighted versions of the above inequalities have been established.
The challenging problem is to find the sharp constant in \eqref{e1.1},
$$
c_{n}=\sup \{ \Vert p' \Vert/\Vert p\Vert \ :\ p\in\PP_n,\ p \ne 0 \}
$$
and in some cases it has been determined explicitly.

When one considers the norm in a Hilbert space the sharp constant $c_{n}$ in Markov's inequality
for polynomials is the largest eigenvalue of a certain matrix. Despite this fact, even in the $L^2$
spaces induced by the classical weight functions of Jacobi
($w_{\alpha,\beta}(x)=(1-x)^{\alpha}(1+x)^{\beta}$, $x\in [-1,1]$,
$\alpha,\beta>-1$), Laguerre ($w_{\alpha}(x)=x^{\alpha}\,e^{-x}$,
$x\in (0,\infty)$) and Hermite $w_{H}(x)=e^{-x^2}$, $x\in
(-\infty,\infty)$), the sharp Markov constants are known only in few cases
(here we do not discuss inequalities relating different norms of a
polynomial and its derivative). For the Laguerre
case $\alpha=0$ P. Tur\'{a}n \cite{PT1960} proved that
$$
c_{n}=\Big(2\sin\frac{\pi}{4n+2}\Big)^{-1}\,.
$$
while in the Hermite case, which is a straightforward one,
$$
c_{n}=\sqrt{2n}
$$
and the Hermite polynomial $H_n(x)$ is the unique (up to a
constant factor) extremal polynomial. In the
case of a constant weight function $w(x)\equiv 1$, $x\in [-1,1]$
(the Legendre case), E. Schmidt \cite{ES1944} proved that, with some
$R\in (-6,13)$,
$$
c_{n}=\frac{(2n+3)^2}{4\pi}\,\Big(1-\frac{\pi^2-3}{3(2n+3)^2}+
\frac{16R}{(2n+3)^4}\Big)\,.
$$
Without any claim for completeness, we mention that bounds for the
best constants in the $L^2$ norms induced by the Laguerre or the
Gegenbauer weight functions are obtained in
\cite{AN2016, ANS2016,  PD1987, PD1991, PD2002, GN2003,  NS2016,
NS2016a}. Regarding the asymptotic behaviour of the best Markov constant
$c_{n}$, we point out that $c_{n}\,$ is $\,O(n^{1/2})$,
$\,O(n)$\, and $O(n^{2})$  as $n\rightarrow\infty$ in the cases of
the $L^2$--norms induced by the Hermite, Laguerre, and Gegenbauer
weight functions, respectively.

Weighted versions of (\ref{BI}) for the so-called weights with
doubling properties were established by G. Mastroianni and V. Totik
\cite{MT} when $1\leq p \leq \infty$ and by T. Erdelyi \cite{Erd}
when $0< p <1$.  Recently D. Lubinsky \cite{Lub14} proved the
weighted analog of (\ref{BEF}) for entire functions of exponential
type, for all $p>0$ and for the same type of doubling weights which,
among others, contain those of the form $(1+x^2)^\alpha$, $\alpha\in
\mathbb{R}$.

In this paper we study a discrete weighted Markov-Bernstein inequality for
real algebraic polynomials. For any pair of parameters $\,(c,\beta)\,$ such that $\,c\in
(0,1)\,$ and $\,\beta>0$, the Meixner inner product and Meixner
norm in $\,\PP\,$ are defined by
\begin{equation}
\label{InP}
\langle f,g\rangle=\langle f,g\rangle_{c,\beta}
:=\sum_{k=0}^{\infty}\frac{(\beta)_k}{k!}\,c^k\,f(k)\,g(k),
\end{equation}
where $(\beta)_k$ is the Poshhammer function, $(\beta)_k=\beta(\beta+1)\cdots
(\beta+k-1)$, $k\in \mathbb{N}$, with $\,(\beta)_0=1$, and
\begin{equation}\label{e1.2}
\Vert f\Vert_{c,\beta}=\langle f, f\rangle_{c,\beta}^{1/2}\,.
\end{equation}
The forward difference (shift) operator $\Delta$ is defined by
$$
\Delta\,f(x):=f(x+1)-f(x)\,,\quad x\in \mathbb{N}_0\,.
$$
The Markov-Bernstein inequality for $\,\PP_n\,$ associated with this
norm is
\begin{equation}\label{e1.3}
\Vert \Delta\,p\Vert_{c,\beta}\leq \gamma_n\,\Vert p\Vert_{c,\beta}\,,\qquad
p\in\PP_n\,.
\end{equation}
We are interested in the best (the smallest possible) constant in
\eqref{e1.3},
\begin{equation}\label{e1.4}
\gamma_n=\gamma_n(c,\beta):=\sup\{\Vert \Delta f\Vert_{c,\beta}\,:\,
f\in\PP_n,\ \Vert f\Vert_{c,\beta}=1\}\,.
\end{equation}

Our main result is the following theorem:

\begin{theorem}\label{Th3}
Let $\,\gamma_n(c,\beta)\,$ be the best constant in Markov-Bernstein
inequality \eqref{e1.3}. Then:
\begin{enumerate}[ (i)~~]
\item
For every $n\in \mathbb{N}$, $\gamma_n(c,1)$ satisfies the inequality
\begin{equation}\label{e1.5}
\gamma_n(c,1)<1+ \frac{1}{\sqrt{c}}\,.
\end{equation}
Moreover,
\begin{equation}\label{e1.6}
\lim_{n\rightarrow\infty}\gamma_n(c,1)=1+ \frac{1}{\sqrt{c}}\,.
\end{equation}
\item
For every fixed $\,n\in \mathbb{N}\,$ and $\,c\in (0,1)$,
$\,\gamma_n(c,\beta)\,$ is a decreasing function of $\beta \in
(0,\infty)\,$.
\item
For every fixed $\,c \in (0,1)\,$ and $\,\beta>0\,$ there exists a
constant  $\,C(c,\beta)>0\,$ such that $\,\gamma_n(c,\beta)\leq
C(c,\beta)\quad \text{ for every }\ n\in \mathbb{N}\,$.
\end{enumerate}
\end{theorem}

\begin{remark}\label{rem1}
Inequality \eqref{e1.3} is discrete for two reasons: the Meixner
norm $\,\|\cdot\|_{c,\beta}\,$ is a ``discrete'' one,  and the
derivative is replaced by the forward difference operator.
Theorem~\ref{Th3}(iii) reveals somewhat unusual phenomenon: while
typically the sharp constants in the Markov-Bernstein inequalities
tend to infinity as $n$ grows, here  the sequence
$\{\gamma_n(c,\beta)\}_{n\in\mathbb{N}}$ is bounded.
\end{remark}

Theorem~\ref{Th3}(iii) follows from a Markov-Bernstein inequality
for a wider set of functions. Let $\,\ell_{2}(c,\beta)\,$ be the
Hilbert space of real valued  functions $\,f\,$ defined on
$\,\mathbb{N}_0\,$ (i.e., sequences $f=(f(0), f(1),\ldots))\,$ for
which $\,\|f\|_{c,\beta}<\infty\,$. We prove the following
Markov-Bernstein inequality for sequences in
$\,\ell_{2}(c,\beta)\,$:
\begin{theorem}  \label{Th2} Let $c \in (0,1)$.
\begin{enumerate}[(i)~~]
\item
If $\,\beta\geq 1$, then
\begin{equation}\label{e1.7}
\| \Delta f \|_{c,\beta} \leq   \Big(1+ \frac{1}{\sqrt{c}}\Big)\,
\| f \|_{c,\beta}\,,\qquad f\in \ell_2(c,\beta)
\end{equation}
and for $\,\beta=1\,$ the constant $\,1+\frac{1}{\sqrt{c}}\,$
cannot be replaced by a smaller one.
\item
If $\,0<\beta\leq 1$, then
\begin{equation}\label{e1.8}
\| \Delta f \|_{c,\beta} \leq   \Big(1+
\frac{1}{\sqrt{\beta\,c}}\Big)\, \| f \|_{c,\beta}\,,\qquad f\in
\ell_2(c,\beta)\,.
\end{equation}
\end{enumerate}
\end{theorem}
Set
$$
\wt{\gamma}(c,\beta):= \sup_{\underset{f\ne 0}{f\in \ell_2(c,\beta)}}\frac{\|\Delta f\|_{c,\beta}}
{\|f\|_{c,\beta}}\,,
$$
then Theorem~\ref{Th2} implies
\begin{equation}\label{e1.9}
\wt{\gamma}(c,\beta)\leq \begin{cases} 1+ \frac{1}{\sqrt{c}}\,,& \beta\geq 1\,,\\
1+ \frac{1}{\sqrt{\beta\,c}}\,,& 0<\beta<1\,.
\end{cases}
\end{equation}
Since $\,\PP_n\subset\ell_2(c,\beta)\,$, we have
$\,\gamma_n(c,\beta)\leq \wt{\gamma}(c,\beta)$, $\,n\in \mathbb{N}$,
hence inequality \eqref{e1.5} in Theorem~\ref{Th3}(i) and
Theorem~\ref{Th3}(iii) are a consequence of \eqref{e1.9}

The rest of the paper is structured as follows. Theorem~\ref{Th2} is
proven in Section~2. In Section~3 we give some properties of the
Meixner polynomials, the orthogonal polynomials with respect to the
Meixner inner product. A relation between the best Markov constants
$\,\gamma_n$, $n\in \mathbb{N}$, and the smallest eigenvalues of
some Jacobi matrices is established in Section~4. It is worth
noticing that this result applies to more general situations,
concerning the sharp constants in a wide class of polynomial
inequalities in $L^2$-norms (see \cite{AN2016, NS2019}). In
Section~5 we obtain two-sided estimates for $\,\gamma_n(c,1)\,$
which complete the proof of Theorem~\ref{Th3}~(i). In Section 6 we
apply the Hellmann-Feynman theorem to prove Theorem~\ref{Th3}~(ii).
Section~7 contains some comments.
\section{Proof of Theorem~\ref{Th2}}
Set $\,\wt{f}(\cdot):=f(\cdot+1)$, then by the triangle inequality
\begin{equation}\label{e4}
\|\Delta f\|_{c,\beta}=\|\wt{f}-f\|_{c,\beta}\leq
\|\wt{f}\|_{c,\beta}+\| f \|_{c,\beta}\,.
\end{equation}
We have
$$
\|\wt{f}\|_{c,\beta}^2=
\sum_{k=0}^{\infty}c^{k}\frac{(\beta)_{k}}{(k)!}[f(k+1)]^2=
\frac{1}{c}\,\sum_{k=1}^{\infty}
c^{k}\frac{(\beta)_{k}}{k)!}[f(k)]^2\,\frac{k}{k-1+\beta}\,.
$$
Since
$$
\frac{k}{k-1+\beta}\leq \begin{cases} 1,& \beta\geq 1\vspace*{2ex}\\
\ds{\frac{1}{\beta}},& 0<\beta\leq 1
\end{cases}
$$
for every $\,k\in \mathbb{N}$, we conclude that
$$
\|\tilde{f}\|_{c,\beta}\leq \begin{cases}\ds{\frac{\|
f\|_{c,\beta}}{\sqrt{c}}},& \beta\geq 1\vspace*{1ex}\\
\ds{\frac{\| f\|_{c,\beta}}{\sqrt{\beta\,c}}},& 0<\beta\leq 1\,.
\end{cases}
$$
By substituting these upper bounds for $\|\wt{f}\|_{c,\beta}$ in the
right-hand side of \eqref{e4} we obtain inequalities \eqref{e1.7}
and \eqref{e1.8}.

It remains to prove the sharpness of the constant $\,1+1/\sqrt{c}\,$
in the case $\beta=1$. For an arbitrary fixed $\,n\in \mathbb{N}\,$
we consider the sequence
$$
f(k)=\begin{cases} \ds{\frac{(-1)^k}{c^{k/2}}}\,,& 0\leq k\leq n
\vspace*{2mm}\\
0\,,& k>n\,.
\end{cases}
$$
We have $\,\| f \|_{c,1}^2=n+1\,$ and
$$
\Delta
f(k)=\begin{cases}\ds{\frac{(-1)^{k+1}}{c^{k/2}}\,
\Big(1+\frac{1}{\sqrt{c}}\Big)}\,,
& 0\leq k\leq n-1 \vspace*{2mm}\\
\ds{\frac{(-1)^{n+1}}{c^{n/2}}}\,,& k=n\vspace*{2mm}\\
0\,,& k>n\,.\end{cases}
$$
Consequently,
$$
\|\Delta f \|_{c,1}^2
=\sum_{k=0}^{n-1}\Big(1+\frac{1}{\sqrt{c}}\Big)^2+1
>n\,\Big(1+\frac{1}{\sqrt{c}}\Big)^2\,.
$$
Hence,
$$
\|\Delta f \|_{c,1}> \Big(\frac{n}{n+1}\Big)^{1/2}\,
\Big(1+\frac{1}{\sqrt{c}}\Big)\,\| f\|_{c,1}
$$
and
$$
\wt{\gamma}(c,1)\geq\lim_{n\to\infty}\Big(\frac{n}{n+1}\Big)^{1/2}\,
\Big(1+\frac{1}{\sqrt{c}}\Big)=1+\frac{1}{\sqrt{c}}\,.
$$
This inequality and \eqref{e1.9} with $\,\beta=1\,$ imply
$\,\wt{\gamma}(c,1)=1+1/\sqrt{c}\,$.
\section{Meixner polynomials}
For any pair of parameters $\,(\beta,c)\,$ such that $\beta>0$ and
$c\in (0,1)$, the Meixner inner product and norm are defined by
\begin{equation*}
\langle f,g\rangle=\langle f,g\rangle_{c,\beta}
:=\sum_{k=0}^{\infty}\frac{(\beta)_k}{k!}\,c^k\,f(k)\,g(k),\qquad
\|f\|_{c,\beta}=\langle f,f\rangle_{c,\beta}^{1/2}\,.
\end{equation*}

The induced Hilbert space $\,\ell_2(c,\beta)=\{f\;:\;
\|f\|_{c,\beta}<\infty\}\,$ contains $\,\PP\,$ and the corresponding
orthogonal polynomials are the Meixner polynomials
$\{M_n(\cdot;\beta,c)\}_{n\in \mathbb{N}_0}$, defined by
$$
M_n(x;\beta,c):=\, _2F_1\Big(\begin{array}{c}-n,-x \\
\beta\end{array}\Big|\,\ds{1-\frac{1}{c}}\Big)\,.
$$
Here, $_2F_1$ is the hypergeometric function,
$$
_2F_1\Big(\begin{array}{c} p,q \vspace*{-1ex}\\
r\end{array}\Big|\,t\Big) = \sum_{k=0}^{\infty}\frac{(p)_k
(q)_k}{(r)_k}\,\frac{t^{k}}{k!}.
$$
 In the following lemma we collect some properties of Meixner
polynomials.
\begin{lemma}\label{l2.1}
The following are properties of Meixner polynomials:
\begin{enumerate}[(i)~~]
\item Orthogonality:
$$
\langle M_m,M_n\rangle:=\sum_{x=0}^{\infty}\frac{(\beta)_x}{x!}
\,c^x\,M_m(x;\beta,c)M_n(x;\beta,c)
=\frac{c^{-n}\,n!}{(\beta)_n(1-c)^{\beta}}\,\delta_{m,n}\,,\quad
m,n\in \mathbb{N}_0\,;
$$
\item Forward shift operator identity:
$$
\Delta
M_n(x;\beta,c):=M_n(x+1;\beta,c)-M_n(x;\beta,c)=\frac{n}{\beta}\,\frac{c-1}{c}\,
M_{n-1}(x;\beta+1,c)\,;
$$
\item  Recurrence relation:
$$
(n+\beta)M_n(x;\beta+1,c)=\beta\,M_n(x;\beta,c)+n\,M_{n-1}(x;\beta+1,c);
$$
\item  Expansion formula:
$$
M_n(x;\beta+1,c)=\frac{n!}{(\beta+1)_n}\,\sum_{k=0}^{n}\frac{(\beta)_k}{k!}
\,M_{k}(x;\beta,c)\,,\quad n\in \mathbb{N}_0\,.
$$
\end{enumerate}
\end{lemma}

\begin{proof} Properties (i) and (ii) are well known, see, e.g., \cite[ (1.9.2),
(1.9.6)]{KS1998}. For the proof of property (iii), we write, with
$z=1-1/c$, the formulae for $M_n(x;\beta+1,c)$ and $M_n(x;\beta,c)$:
\begin{eqnarray*}
&&M_n(x;\beta+1,c)=1+\sum_{k=1}^{n}\binom{n}{k}\,\frac{x(x-1)\cdots(x-k+1)}
{(\beta+1)(\beta+2)\cdots(\beta+k)}\,z^k\,,\\
&&M_n(x;\beta,c)=1+\sum_{k=1}^{n}\binom{n}{k}\,\frac{x(x-1)\cdots(x-k+1)}
{\beta(\beta+1)\cdots(\beta+k-1)}\,z^k\,.
\end{eqnarray*}
Subtracting the second equality multiplied by $\beta$ from the first
one multiplied by $n+\beta$, we obtain the result.

The proof of property (iv) is by induction with respect to $n$.
Obviously, the equality holds for $n=0$, and we assume it is true
for some $n\in \mathbb{N}_0$. Property (iii) and the inductional
hypothesis then imply
\[
\begin{split}
M_{n+1}(x;\beta+1,c)&=\frac{\beta}{n+1+\beta}\,M_{n+1}(x;\beta,c)
+\frac{n+1}{n+1+\beta}\,M_{n}(x;\beta+1,c)\\
&=\frac{\beta}{n+1+\beta}\,M_{n+1}(x;\beta,c)+
\frac{(n+1)!}{(\beta+1)_{n+1}}\,
\sum_{k=0}^{n}\frac{(\beta)_k}{k!} \,M_{k}(x;\beta,c)\\
&=\frac{(n+1)!}{(\beta+1)_{n+1}}\,
\sum_{k=0}^{n+1}\frac{(\beta)_k}{k!} \,M_{k}(x;\beta,c)\,,
\end{split}
\]
which accomplishes the induction step.
\end{proof}

In view of Lemma~\ref{l2.1}\,(i), the orthonormal Meixner
polynomials $\{p_m\}_{m\in \mathbb{N}_0}$ are given by
\begin{equation}\label{e2.1}
p_m(x;\beta,c) :=(1-c)^{\frac{\beta}{2}}\,c^{\frac{m}{2}}\,
\sqrt{\frac{(\beta)_m}{m!}}\,M_m(x;\beta,c)\,.
\end{equation}

The \emph{forward shift operator} of the orthonormal Meixner
polynomials obeys the following representation:
\begin{lemma}\label{l2.2}
For any $m\in \mathbb{N}$,
$$
\Delta p_m(x;\beta,c)=
\frac{c-1}{c}\,\sum_{k=0}^{m-1}\frac{\alpha_k}{\alpha_m}\,p_k(x;\beta,c)\,,
$$
where
\begin{equation}\label{e2.2}
\alpha_k:=c^{-\frac{k}{2}}\,\sqrt{\frac{(\beta)_k}{k!}}\,.
\end{equation}
\end{lemma}

\begin{proof}
From Lemma~\ref{l2.1}(ii), (iv) we have
\[
\begin{split}
\Delta\,M_m(x;\beta,c)&=\frac{m}{\beta}\,\frac{c-1}{c}\,
\frac{(m-1)!}{(\beta+1)_{m-1}}\,\sum_{k=0}^{m-1}
\frac{(\beta)_k}{k!}\,M_k(x;\beta,c)\\
&=\frac{c-1}{c}\,\frac{m!}{(\beta)_m}\,\sum_{k=0}^{m-1}
\frac{(\beta)_k}{k!}\,M_k(x;\beta,c)\,,
\end{split}
\]
or, equivalently,
$$
\frac{(\beta)_m}{m!}\,\Delta\,M_m(x;\beta,c)
=\frac{c-1}{c}\,\sum_{k=0}^{m-1}
\frac{(\beta)_k}{k!}\,M_k(x;\beta,c)\,.
$$
In this identity we replace $M_k(x;\beta,c)$ by
$$
M_k(x;\beta,c)=c^{-\frac{k}{2}}(1-c)^{-\frac{\beta}{2}}
\sqrt{\frac{k!}{(\beta)_k}}\,p_k(x;\beta,c)\,,\quad k=0,\ldots,m,
$$
and deduce the desired representation.
\end{proof}

\section{Best Markov constants and extreme eigenvalues of Jacobi matrices}

In seeking for the best Markov constant
\begin{equation}\label{e3.1}
\gamma_n(c,\beta):=\sup\{\Vert \Delta f\Vert_{c,\beta}\,:\, f\in\PP_n,\
\Vert f\Vert_{c,\beta}=1\}\,,
\end{equation}
we may assume without loss of generality that
$$
f=t_1 p_1+t_2 p_2+\cdots +t_n p_n
=\mathbf{t}^{\top}\mathbf{p}_1\,,\qquad \Vert f\Vert=1=\Vert
\mathbf{t}\Vert =(t_1^2+\cdots+t_n^2)^{1/2}=1\,,
$$
with $\mathbf{t}^{\top}=(t_1,t_2,\ldots,t_n)\in \mathbb{R}^n$ and
$\mathbf{p}_1^{\top}=(p_1,p_2,\ldots,p_n)$. Indeed, since
$\Delta\,p_0=0$, $p_0$ cannot yield an increase of $\Vert \Delta
f\Vert_{c,\beta}$.\smallskip

According to Lemma~\ref{l2.2}, we have
$$
\Delta \mathbf{p}_1=\frac{c-1}{c}\,\mathbf{A_n}\,\mathbf{p}_0\,
$$
where $\mathbf{p}_0^{\top}:=(p_0,p_1,\ldots,p_{n-1})$ and
\begin{equation}\label{e3.2}
\mathbf{A_n}:=\begin{pmatrix}
\frac{\alpha_0}{\alpha_1}&0&0&\cdot\cdots&0\\
\frac{\alpha_0}{\alpha_2} &\frac{\alpha_1}{\alpha_2}&0&\cdots&0\\
\frac{\alpha_0}{\alpha_3} &\frac{\alpha_1}{\alpha_3} &
\frac{\alpha_2}{\alpha_3} & \cdots&0\\
\vdots&\vdots&\vdots&\ddots&\vdots\\
\frac{\alpha_0}{\alpha_n} & \frac{\alpha_1}{\alpha_n}
&\frac{\alpha_2}{\alpha_n} & \cdots& \frac{\alpha_{n-1}}{\alpha_n}
\end{pmatrix}\,.
\end{equation}

Hence, $\Delta f=(1-1/c)\,\mathbf{t}^{\top}\Delta
\mathbf{p}_1=(1-1/c)\,\mathbf{t}^{\top} \mathbf{A_n}\,\mathbf{p}_0$,
and
$$
\Vert \Delta f\Vert_{c,\beta}^2 =(1-1/c)^2\Vert \mathbf{t}^{\top}
\mathbf{A_n}\Vert^2=(1-1/c)^2\langle
\mathbf{A_n^\top}\mathbf{t}, \mathbf{A_n^{\top}} \mathbf{t}\rangle
=(1-1/c)^2\langle
\mathbf{A_n}\mathbf{A_n^\top}\mathbf{t},\mathbf{t}\rangle.
$$
Therefore,
\begin{equation}\label{e3.3}
\gamma_n^2(c,\beta)=(1-1/c)^2\, \sup_{\Vert
\mathbf{t}\Vert=1}\langle
\mathbf{A_n}\mathbf{A_n^{\top}}\mathbf{t},\mathbf{t}\rangle=
(1-1/c)^2\,\mu_{\max}(c,\beta),
\end{equation}
where $\mu_{\max}=\mu_{\max}(c,\beta)$ is the largest eigenvalue of
the positive definite matrix $\mathbf{A_n}\mathbf{A_n^{\top}}$.

Since $\,\mathbf{A_n^{\top}}\mathbf{A_n}=
\mathbf{A_n}^{-1}(\mathbf{A_n}\mathbf{A_n^{\top}})\mathbf{A_n}\,$,
$\,\mathbf{A_n^{\top}}\mathbf{A_n}\sim
\mathbf{A_n}\mathbf{A_n^{\top}}$, and therefore $\,\mu_{\max}\,$ is
also the largest eigenvalue of the positive definite matrix
$\,\mathbf{A_n^{\top}}\mathbf{A_n}\,$.

It turns out that it is advantageous to work with the inverse
matrices
$$
\mathbf{B_n}=(\mathbf{A_n}\mathbf{A_n^{\top}})^{-1}\,,\qquad
\mathbf{C_n}=(\mathbf{A_n^{\top}}\mathbf{A_n})^{-1}\,,
$$
as we shall show that they are Jacobi matrices.

Let us we find the explicit form of $\mathbf{B_n}$ and
$\mathbf{C_n}$. The matrix $\mathbf{A_n}$ in \eqref{e3.1} can be
represented in the form
\begin{equation}\label{e3.4}
\mathbf{A_n}={\rm diag}\{\alpha_k^{-1}\}\, \mathbf{T_n}\, {\rm
diag}\big\{\alpha_{k-1}\big\}\,,
\end{equation}
where ${\rm diag}\{\alpha_k^{-1}\}$ and ${\rm diag}\{\alpha_{k-1}\}$
are diagonal $n\times n$ matrices with entries on the main diagonal
$(1/\alpha_1,\ldots,1/\alpha_n)$ and
$(\alpha_0,\alpha_1,\ldots,\alpha_{n-1})$, respectively, and
$\mathbf{T_n}$ is an $n\times n$ triangular matrix with entries
$t_{i,j}=1$, if $i\ge j$, and $t_{i,j}=0$, otherwise.

The matrices $\mathbf{T_n}^{-1}$ and $(\mathbf{T_n^{\top}})^{-1}$
are two-diagonal, namely the only nonzero entries of
$\mathbf{T_n}^{-1}$ are $t^{-1}_{k,k}=1$, $k=1,\ldots,n$, and
$t^{-1}_{k+1,k}=-1$, $k=1,\ldots,n-1$. It follows from \eqref{e3.4}
that
$$
\mathbf{B_n}=(\mathbf{A_n}\mathbf{A_n^{\top}})^{-1}={\rm
diag}\{\alpha_k\}\,(\mathbf{T_n^{\top}})^{-1}\,{\rm
diag}\{\alpha_{k-1}^{-2}\} \,\mathbf{T_n}^{-1}\,  {\rm
diag}\{\alpha_k\}\,,
$$
and using the explicit form of $\mathbf{T_n}^{-1}$ and
$(\mathbf{T_n^{\top}})^{-1}$, we conclude that $\mathbf{B_n}$ is a tridiagonal matrix
whose diagonal entries are $b_{k,k}=1+ \alpha_k^2/\alpha_{k-1}^2$, $k=1,\ldots,n-1$, and $b_{n,n}=\alpha_n^2/\alpha_{n-1}^2$,
while the off-diagonal ones are $b_{k,k+1}=b_{k+1,k}=-\alpha_{k+1}/\alpha_k$, $k=1,\ldots, n-1$.

In a similar manner, \eqref{e3.4} implies
$$
\mathbf{C_n}=(\mathbf{A_n^{\top}}\mathbf{A_n})^{-1}={\rm
diag}\{\alpha_{k-1}^{-1}\}\, \mathbf{T_n}^{-1}\,{\rm
diag}\{\alpha_{k}^{2}\}\,(\mathbf{T_n^{\top}})^{-1}\,{\rm
diag}\{\alpha_{k-1}^{-1}\}\,
$$
so that that $\mathbf{C_n}$ is a tridiagonal matrix
whose diagonal entries are $c_{1,1}=\alpha_1^2/\alpha_{0}^2$ and $c_{k,k}=1+ \alpha_k^2/\alpha_{k-1}^2$, $k=2,\ldots,n$,
and the off-diagonal ones are $c_{k,k+1}=c_{k+1,k}=-\alpha_{k}/\alpha_{k-1}$, $k=1,\ldots, n-1$.
Replacement of the explicit values of $\alpha_k$ from \eqref{e2.2} yields
\begin{equation}\label{e3.5}
\mathbf{B_n}=\begin{pmatrix} \frac{\beta}{c}+1 &
-\sqrt{\frac{\beta+1}{2c}} & 0 & & \cdots & 0 & 0 \\
-\sqrt{\frac{\beta+1}{2c}} & \frac{\beta+1}{2c}+1 &
-\sqrt{\frac{\beta+2}{3c}} & & \cdots & 0 & 0 \\ 0 &
-\sqrt{\frac{\beta+2}{3c}} & \frac{\beta+2}{3c}+1 & &
\cdots & 0 & 0 \\ & & & & & \\
\vdots & \vdots & \vdots  & & \ddots & \vdots & \vdots\\
& & & & & \\
0 & 0 & 0  & & \cdots & \frac{\beta+n-2}{(n-1)c}+1 &
-\sqrt{\frac{\beta+n-1}{n\,c}} \\ 0 & 0 & 0  & & \cdots &
-\sqrt{\frac{\beta+n-1}{n\,c}} & \frac{\beta+n-1}{n\,c}
\end{pmatrix},
\end{equation}
\begin{equation}\label{e3.6}
\mathbf{C_n}=\begin{pmatrix} \frac{\beta}{c} &
-\sqrt{\frac{\beta}{c}} & 0 & & \cdots & 0 & 0 \\
-\sqrt{\frac{\beta}{c}} & \frac{\beta+1}{2c}+1 &
-\sqrt{\frac{\beta+1}{2c}} & & \cdots & 0 & 0 \\ 0 &
-\sqrt{\frac{\beta+1}{2c}} & \frac{\beta+2}{3c}+1 & &
\cdots & 0 & 0 \\ & & & & & \\
\vdots & \vdots & \vdots  & & \ddots & \vdots & \vdots\\
& & & & & \\
0 & 0 & 0  & & \cdots & \frac{\beta+n-2}{(n-1)c}+1 &
-\sqrt{\frac{\beta+n-2}{(n-1)c}} \\ 0 & 0 & 0  & & \cdots &
-\sqrt{\frac{\beta+n-2}{(n-1)c}} & \frac{\beta+n-1}{n\,c}+1
\end{pmatrix}
\end{equation}

Let $\mathbf{\tilde{B}_n}$ and $\mathbf{\tilde{C}_n}$ be the
corresponding Jacobi matrices whose diagonal entries coincide with
those of $\mathbf{B_n}$ and $\mathbf{C_n}$ but the off-diagonal ones
are opposite to those of $\mathbf{B_n}$ and $\mathbf{C_n}$. Then
obviously the eigenvalues of  $\mathbf{B_n}$ and
$\mathbf{\tilde{B}_n}$ coincide and those of $\mathbf{C_n}$ and
$\mathbf{\tilde{C}_n}$  also do.

Since $\mu_{\max}=1/\lambda_{\min}$, where $\lambda_{\min}$ is the
smallest eigenvalue of either of  the matrices $\mathbf{B_n}$,
$\mathbf{\tilde{B}_n}$,  $\mathbf{C_n}$ and $\mathbf{\tilde{C}_n}$,
\eqref{e3.3} yields the following
\begin{theorem}\label{t3.1}
The best constant $\gamma_n(c,\beta)$ in the Markov-Bernstein inequality
$$
\Vert \Delta\,p\Vert_{c,\beta}\leq \gamma_n\,\Vert p\Vert_{c,\beta}\,,\qquad
p\in\PP_n
$$
admits the representation
\begin{equation}\label{e3.7}
\gamma_n(c,\beta)=\frac{1/c-1}{\sqrt{\lambda_{\min}(c,\beta)}}\,,
\end{equation}
where $\lambda_{\min}(c,\beta)>0$ is the smallest eigenvalue of
either of the matrices $\mathbf{B_n}$, $\mathbf{\tilde{B}_n}$, $\mathbf{C_n}$ and $\mathbf{\tilde{C}_n}$.
\end{theorem}

As is well-known, every $n\times n$ Jacobi matrix
$\mathbf{J_n}$ defines through a three term recurrence relation a
sequence of orthonormal polynomials $\{P_m\}_{m=0}^{n}$, and the
zeros of $P_n$ are the eigenvalues of $\mathbf{J_n}$. We therefore
may reformulate Theorem~\ref{t3.1} as
\medskip

\noindent {\bf Theorem 4.1$^{\prime}$~~} {\it The best Markov
constant $\gamma_n(c,\beta)$ admits the representation \eqref{e3.7},
where $\lambda_{\min}(c,\beta)$ is the smallest zero of the $n$-th
polynomial  $P_n=P_n(c,\beta;\cdot)$ in the sequence of polynomials
defined recursively by
\begin{eqnarray*}
&& P_{0}(x)=1\,,\ \ P_{1}(x)=x-\frac{\beta}{c}\,,\\
&&P_k(x)=\Big(x-\frac{\beta+k-1}{k\,c}-1\Big)\,P_{k-1}(x)
-\frac{\beta+k-2}{(k-1)c}\,P_{k-2}(x)\,,\quad k\geq 2\,.
\end{eqnarray*}}
Since $\beta>0$ and $c\in (0,1)$, it follows from Favard's theorem
that $\{P_k\}_{k\in \mathbb{N}_0}$ form a system of orthogonal
polynomials.
\section{Two-sided estimates for $\,\gamma(c,1)$}
Matrices $\mathbf{B}=\mathbf{B_n}(c,\beta)$ and
$\mathbf{C}=\mathbf{C_n}(c,\beta)$ have particularly simple form in
the case $\beta=1$, for instance,
$$
\mathbf{D_n} := \mathbf{C_n}(c,1)=\begin{pmatrix} \frac{1}{c} &
-\frac{1}{\sqrt{\,c}} & 0 & & \cdots & 0 & 0 \\
-\frac{1}{\sqrt{\,c}} & \frac{1}{c}+1 & -\frac{1}{\sqrt{\,c}} & &
\cdots & 0 & 0 \\ 0 & -\frac{1}{\sqrt{\,c}} & \frac{1}{c}+1 & &
\cdots & 0 & 0 \\
\vdots & \vdots & \vdots & & \ddots & \vdots & \vdots\\
0 & 0 & 0 & & \cdots &  \frac{1}{c}+1 & -\frac{1}{\sqrt{\,c}} \\
0 & 0 & 0 & & \cdots & -\frac{1}{\sqrt{\,c}} & \frac{1}{c}+1
\end{pmatrix} .
$$

We shall find estimates for $\,\lambda_{\min}(c,1)$, the smallest
zero of $|\lambda\,\mathbf{E_n}-\mathbf{D_n}|=0$. By change of
variable
$$
\lambda=1+\frac{1}{c}+\frac{2z}{\sqrt{c}}
$$
this equation simplifies to $\varphi_n(z)=0$, where
$$
\varphi_n(z)=
\begin{vmatrix} z+\frac{\sqrt{c}}{2} &
\frac{1}{2} & 0 & & \cdots & 0 & 0 \\
\frac{1}{2} & z & \frac{1}{2} & & \cdots & 0 & 0 \\ 0 & \frac{1}{2}
& z & & \cdots & 0 & 0 \\
\vdots & \vdots & \vdots & & \ddots & \vdots & \vdots\\

0 & 0 & 0 & & \cdots &  z & \frac{1}{2} \\
0 & 0 & 0 & & \cdots & \frac{1}{2} & z
\end{vmatrix}\,.
$$
It is easy to see that
\begin{equation}\label{e5.1}
\varphi_n(z)=\frac{1}{2^n}\,\big(U_n(z)+\sqrt{c}\,U_{n-1}(z)\big)\,,
\end{equation}
where $U_m(z)$ is the $m$-th Chebyshev polynomial of second kind,
$$
U_m(z)=\frac{\cos (m+1)\arccos(z)}{\sqrt{1-z^2}},\ \ \,z\in [-1,1].
$$
Indeed, \eqref{e5.1} is readily verified to be true for $n=1,\,2$,
and $\{\varphi_m\}$ satisfy the recurrence relation
$$
\varphi_m(z)=z\,\varphi_{m-1}(z)-\frac{1}{4}\,\varphi_{m-2}(z),\qquad
m\geq 3,
$$
which is also satisfied by $\{2^{-m}U_m\}$\,.

\begin{lemma}\label{l5.1}
The zeros of $\varphi_n$, $\,n\geq 2$, are located in $(-1,1)$ and
interlace with the zeros of $U_{n-1}$. Moreover, if $\tau$ is the
smallest zero of $\varphi_n$, then
$$
\tau=-1+\varepsilon_n\,,\qquad  \frac{1}{n(n+1)}<\varepsilon_n<
2\,\sin^2\frac{\pi}{2n}\,.
$$
\end{lemma}

\begin{proof} Clearly, $\varphi_n$ is a monic polynomial of degree $n$.
Let $\eta_k=\cos \frac{k\pi}{n}$, $\,k=1,\ldots,n-1$, be the zeros
of $U_{n-1}$, then $-1<\eta_{n-1}<\eta_{n-2}<\cdots<\eta_1<1$, and
$$
\sgn \varphi(\eta_k)=\sgn U_n(\eta_k)=(-1)^{k}\,,\qquad k=1,\ldots,
n-1.
$$
This and
$\,\varphi_n(-1)=(-1)^{n}\,2^{-n}\,\big(n+1-n\,\sqrt{c}\big)$,
$\,\varphi_n(1)=2^{-n}\,\big(n+1+n\,\sqrt{c}\big)\,$ imply that the
zeros of $\,\varphi_n\,$ lie in $\,(-1,1)\,$ and interlace with the
zeros of $\,U_{n-1}$.

The upper bound for $\varepsilon_n$ follows from
$$
\tau<\eta_{n-1}=-\cos\frac{\pi}{n}=-1+2\sin^2\frac{\pi}{2n}\,.
$$
To obtain the lower bound for $\varepsilon_n$, we apply one step of
Newton's method for finding $\tau$ as the smallest zero of
$\varphi_n(z)$ with initial value $\tau^{(0)}=-1$. We have
$$
\tau>\tau^{(1)}=-1+\frac{\varphi_n(-1)}{\varphi_n^{\prime}(-1)}
=-1+\frac{3(n+1-n\sqrt{c})}{n(n+1)\big(n+2-(n-1)\sqrt{c}\big)}>
-1+\frac{1}{n(n+1)}\,,
$$
where for the last inequality we have used that
$g(x)=\frac{n+1-n\,x}{n+2-(n-1)\,x}$ is a decreasing function in
$(0,1)$.
\end{proof}

Going back to variable $\lambda$, we find
$$
\lambda_{\min}(c,1)=1+\frac{1}{c}+\frac{2\tau}{\sqrt{c}}
=1+\frac{1}{c}+\frac{2(-1+\varepsilon_n)}{\sqrt{c}}
=\Big(\frac{1}{\sqrt{c}}-1\Big)^2+\frac{2\varepsilon_n}{\sqrt{c}}\,,
$$
hence
\begin{equation}\label{e5.2}
\frac{1}{\lambda_{\min}(c,1)}=\frac{1}{\Big(\dfrac{1}
{\sqrt{c}}-1\Big)^2\Big(1+\dfrac{2\sqrt{c}\,\varepsilon_n}{(1-\sqrt{c})^2}\Big)}\,.
\end{equation}
Now \eqref{e3.7}, \eqref{e5.2} and the estimates for $\varepsilon_n$
from Lemma~\ref{l5.1} imply
\begin{theorem}\label{t5.2}
For any $n\geq 2$, the best constant $\,\gamma_n(c,1)\,$ in the
Markov-Bernstein inequality
$$
\Vert \Delta\,p\Vert_{c,1}\leq \gamma_n(c,1)\,\Vert
p\Vert_{c,1}\,,\qquad p\in\PP_n\,,
$$
admits the estimates
\begin{equation}\label{e5.3}
\frac{1+\dfrac{1}{\sqrt{c}}}{\Big(1+\dfrac{4\sqrt{c}}
{(1-\sqrt{c})^2}\sin^2\frac{\pi}{2n}\Big)^{1/2}}\leq \gamma_n(c,1)
\leq \frac{1+\dfrac{1}{\sqrt{c}}}{\Big(1+\dfrac{2\sqrt{c}}
{(1-\sqrt{c})^2\,n(n+1)}\Big)^{1/2}}\,.
\end{equation}
\end{theorem}
Theorem~\ref{Th3}(i) now follows from the two-sided estimates
\eqref{e5.3}. Note that the upper estimate for $\,\gamma_n(c,1)\,$
in \eqref{e5.3} sharpens the one in Theorem~\ref{Th3}~(i).
\section{Monotone dependence of eigenvalues on $\beta$}

The statement of Theorem~\ref{Th3}\ (ii) is a consequence of the
following

\begin{proposition}
\label{Prop5.1}
For a fixed $c\in (0,1)$, each eigenvalue $\lambda$ of the matrix
$\mathbf{B_n}(\beta,c)$, defined by \eqref{e3.5}, is a strictly
monotone increasing function of $\,\beta\,$ in the interval
$\,(0,\infty)$.
\end{proposition}

We apply the elegant method to establish monotonicity of zeros  of
orthogonal polynomials, or equivalently of eigenvalues of Jacobi
matrices based on Hellmann-Feynman's theorem \cite{Hel, Fey} and
Wall-Wetzel's criterion \cite{WV} for positive definiteness  of
Jacobi matrices. We describe it briefly and refer to Chapter 7.3 in
Ismail's book \cite{IsmBook} as well as to \cite{Ism87,  IsmMul,
IsmZ} for more details. Consider the parametric sequence
$\{p_k(x;\tau)\}_{k=0}^\infty$ of orthonormal  polynomials which is
generated by the three term recurrence relation
\begin{eqnarray*}
p_{-1}(x;\tau) & = & 0,\\
p_0(x;\tau) & = & 1,\\
x\, p_k(x;\tau) & = & a_{k}(\tau)\, p_{k+1}(x;\tau) + b_{k}(\tau)\,
p_k(x;\tau) + a_{k-1}(\tau)\, p_{k-1}(x;\tau), \ \  k\geq0,
\end{eqnarray*}
where $a_{k-1}(\tau)>0$. The zeros of the polynomial $p_n(x;\tau)$
coincide with the  eigenvalues of the Jacobi matrix $\mathbf{J_n} =
\mathbf{J_n}(\tau)$, whose diagonal entries are $b_k(\tau)$, $k=0,
\ldots, n-1$, and the off-diagonal ones are $a_k(\tau)$, $k=0,
\ldots, n-2$. Moreover, if $\lambda_j=\lambda_j(\tau)$ is a zero of
$p_n(x;\tau)$ and
\[
\mathbf{p_j} = (p_0(\lambda_j;\tau), p_1(\lambda_j;\tau),  \ldots,
p_{n-1}(\lambda_j;\tau))^{\top},
\]
then
\[
\mathbf{J_n}\, \mathbf{p_j} = \lambda_j\, \mathbf{p_j}.
\]
Let us denote by $\mathbf{J_n^\prime}=\mathbf{J_n^\prime}(\tau)$ the
tridiagonal matrix whose entries are the derivatives of the
corresponding entries of $\mathbf{J_n}(\tau)$. Then the
Hellmann-Feynman theorem, in the particular case which is convenient
for our objectives, reads as follows:

\begin{theorem}\label{HFtheo}
For every zero $\lambda_j(\tau)$ of $p_n(x;\tau)$ we have
\[
\lambda_{j}^{\prime}(\tau) =  \frac{\mathbf{p_j}^{\top}
\mathbf{J_n}^\prime \mathbf{p_j} } {\mathbf{p_j}^{\top} \mathbf{p_j}
}.
\]
Furthermore, if the numerator of the latter expression is positive,
then the zeros $\lambda_j(\tau)$ of $p_{n}(x;\tau)$ are increasing
functions of $\tau$. In particular, the latter statement holds if
$\mathbf{J_n^\prime}$ is a positive definite matrix.
\end{theorem}

Let us recall that a sequence $\{c_n\}_1^\infty$ of non-negative
numbers is called a chain sequence if there exists another sequence
$\{v_n\}_0^\infty$, called a parametric one, such that $0 \leq v_0 <
1$, $0 < v_n < 1$ for all $n \in \mathbb{N}$, and $c_n= (1-v_{n-1})
v_n$ for every $n \in \mathbb{N}$ (see \cite{IsmBook}). A criterion
for positive definiteness of Jacobi matrices, due to Wall and Wetzel
\cite{WV}, applied to $\mathbf{J_n^\prime}$ yields:
\begin{proposition}\label{propcos}
Let $\mathbf{J_{n}^\prime}$ be a Jacobi matrix with positive
diagonal  entries. If $b_i^\prime>0$ for $i=0,\ldots,n-1$, and
there is a chain sequence $\{ \kappa_i \}$ such that
\[
\frac{ [a_{i}^\prime]^{2} }{ b_{i}^\prime b_{i+1}^\prime  }  <
\kappa_i,\ \ \mathrm{for}\ \ \ i=0,1,\dots, n-2,
\]
then $\mathbf{J_{n}^\prime}$ is positive definite.
\end{proposition}

{\em Proof of Proposition} \ref{Prop5.1}. All we need is to show
that the matrix $\mathbf{\tilde{B}_n^\prime}=
(\tilde{b}_{ij}^\prime)_{n\times n}$ obtained from
$\mathbf{\tilde{B}_n}$ by partial differentiation of its entries
with respect to $\beta$ is positive definite. Straightforward
calculations show that $\tilde{b}_{k,k}^\prime=1/(k\,c)$,
$k=1,\ldots, n$, and
$$
\tilde{b}_{k,k+1}^\prime = \frac{1}{2\sqrt{(k+1)(k+\beta)c}},\,  \ \
k=1,\ldots, n-1.
$$
It follows by Proposition~\ref{propcos} and the fact that that
$\{\kappa_i \}=\{ 1 /4,1/4, \ldots\}$ is a chain sequence that a
sufficient condition for $\mathbf{\tilde{B}_n^\prime}$ to be
positive definite is that the following inequalities are satisfied:
\begin{equation}\label{e4.1}
\frac{[\tilde{b}_{k,k+1}^\prime]^2}{\tilde{b}_{k,k}^\prime\,
\tilde{b}_{k+1,k+1}^\prime}<\frac{1}{4}\,,\qquad k=1,\ldots,n-1\,.
\end{equation}
They are equivalent to inequalities
$$
\beta+ k (1-c)>0\,,\qquad k=1,\ldots,n-1\,,
$$
which are obviously true, as $\beta>0$ and $c\in (0,1)$. Hence,
$\mathbf{\tilde{B}_n^\prime}$ is a positive definite matrix.
\qed

In particular, the smallest eigenvalue of $\,\tilde{\mathbf{B}}\,$,
$\lambda_{\min}(c,\beta)$, is a monotone increasing function of
$\beta$. By Theorem~\ref{t3.1},
$\,\gamma_n(c,\beta)=(1/c-1)/\sqrt{\lambda_{\min}(c,\beta)}\,$ is a
monotone decreasing function of $\beta$, which proves
Theorem~\ref{Th3}~(ii).


\section{Comments}
The Markov-Bernstein inequality for sequences, Theorem \ref{Th2},
was rather easy to prove, and then was used in the proof of parts
(i) and (iii) of Theorem~\ref{Th3}, the Markov-Bernstein inequality
for polynomials. Since we observe, at least for $\,\beta=1$,
coincidence of the best Markov constant in $\ell_2(c,\beta)$ with
the limit of $\,\gamma_n\,$ in the polynomial case, a natural
question is whether, on the contrary, Markov-Bernstein inequality
for sequences can be deduced from Markov-Bernstein inequality for
polynomials. Such an approach would be possible if the corresponding
general Bernstein's problem had a solution in this particular case.
Indeed, let us suppose that the following question has an
affirmative answer:  Is it true that, given $c\in (0,1)$, $\beta \in
[1,\infty)$, a sequence $f\in \ell_{2}(c,\beta)$ and
$\varepsilon>0$, there is an algebraic polynomial $p$, such that
$$
\sum_{k=0}^\infty  \frac{(\beta)_k}{k!}\, c^k\, [f(k)-p(k)]^2 < \varepsilon?
$$
Then the Markov-Bernstein inequality for sequences would be an
immediate consequence of the Markov-Bernstein inequality for
polynomials.

A general version of Bernstein's approximation problem reads as
follows:  given  a weight function $W: \mathbb{R}\to [0,1]$ and the
corresponding weighted norm is $\| \cdot \|_W$, say
$L_W^p(\mathbb{R})$, such that the corresponding moments are finite
in the norm, i.e. $\| P \|_W < \infty$ for every $P\in \PP$, is it
true that for every function $f$ with $\| f \|_W < \infty$ and each
$\varepsilon>0$, there is an algebraic polynomial $P$ such that $\|
f-P \|_W < \varepsilon$? The weights $W$ for which this problem has
an affirmative answer are sometimes called admissible ones. The
first results concerning characterisation of the admissible kernels
for the uniform norm were obtained by S. N. Mergelyan, N. I.
Akhiezer, H. Pollard,  S. Izumi and T. Kawata, M. Dzrbasjan and L.
Carleson. We refer to Lubisnky's  survey \cite{Lub} and P. Koosis's
\cite{Koo} and M. Ganzburg's \cite{Gan2008} books for details,
further contributions and references.

The result in this direction which is the most relevant in our
situation is due to G. Freud \cite[Theorem 3.3. on p. 73]{Fre} (see
also \cite[Theorem 1.7]{Lub}). His contribution is an extension of
M. Riesz' one \cite{MR2} which was obtained even before Bernstein
posed his problem. G. Freud's result applies to our problem but only
for sequences in $\ell_{2,c,\beta}$ of at most polynomial growth,
despite that it implies that such sequence would posses one-sided
polynomial approximations.

We rase also a problem which is the ``continuous'' counterpart of
the one stated above. More precisely, it would be of interest to
know if $W(x) = e^{-ax} \Gamma(x+\beta)/\Gamma(x+1)$, where $a>0$
and $\beta >1$, is an admissible weight for Bernstein's
approximation problem in $L_W^p(0,\infty)$, for $p\geq 1$.  The
above mentioned results of S. Izumi and T. Kawata \cite{IK1937}, M.
Dzrbasjan \cite{Dzr1947}  and L. Carleson \cite{Car1951}  imply that
this is true for $\beta=1$.

Finally, it would be interest to study the eventual  extensions  of
the results in Theorems \ref{Th2} and \ref{Th3} for the relevant
weighted $\ell_p$ norms.

\section*{Acknowledgments}
The first author thanks Doron Lubinsky for his interest in and
valuable comments about the discrete version of Bernstein's
approximation problem stated in the last section. The second author
thanks Vilmos Totik for the fruitful discussions on the results in
this paper. Both authors are grateful to the anonymous referee of
the previous version of the paper, whose careful reading and
valuable suggestions contributed to improvement of the presentation.

\end{document}